\documentclass{elsarticle}

\usepackage{graphicx}
\usepackage[utf8]{inputenc}
\usepackage{subfigure}
\usepackage{amsmath}
\usepackage{amsthm}
\usepackage{amssymb, mathtools, hyperref}
\usepackage{graphicx}
\usepackage{cleveref}
\usepackage{tikz}
\usetikzlibrary{patterns,shapes}
\usepackage{gastex}
\usepackage{todonotes}
\usepackage{cite}

\def\N{\mathbb N}
\def\A{\mathcal A}

\def\uu{\mathbf{u}}

\theoremstyle{definition}
\newtheorem{definition}{Definition}
\newtheorem{corollary}[definition]{Corollary}
\newtheorem{remark}[definition]{Remark}
\newtheorem{example}[definition]{Example}

\theoremstyle{plain}
\newtheorem{theorem}[definition]{Theorem}

\begin{document}
\begin{frontmatter}
\title{String attractors of episturmian sequences}

\author[cvut]{L\!'ubom\'ira Dvo\v r\'akov\'a}
\ead{lubomira.dvorakova@fjfi.cvut.cz}
\address[cvut]{FNSPE Czech Technical University in Prague, Czech Republic}

\begin{abstract}
In this paper, we describe string attractors of all factors of episturmian sequences and show that their size is equal to the number of distinct letters contained in the factor.  
\end{abstract}

\begin{keyword}
episturmian sequences \sep Sturmian sequences \sep string attractors 

\MSC 68R15
\end{keyword}
\end{frontmatter}

\section{Introduction}
String attractors play an important role in the field of data compression. 
They have been recently defined and studied by Kempa and Prezza~\citep{KempaPrezza2018}: a {\em string attractor} of a finite word $w=w_0w_1\dots w_n$, where $w_i$ are letters, is a subset $\Gamma$ of $\{0,1,\dots, n\}$ such that each factor of $w$ has an occurrence containing an element of $\Gamma$. 

In this paper we consider string attractors from the point of view of combinatorics on words. 
Several results have already been reached in this aspect. Basic combinatorial properties of string attractors (e.g., string attractors of powers, conjugates, etc.) were studied by Mantaci et al.~\citep{Mantaci2021}.  
Moreover, the authors showed that finite standard Sturmian words have an attractor of size 2 containing two consecutive positions and that the size of the smallest string attractor of de Bruijn words grows asymptotically as $\frac{n}{\log n}$, where $n$ is the length of the word. The smallest string attractor (of size 4) of a particular factor subset of the Thue-Morse sequence was determined in~\citep{Kutsukake2020}.
Schaeffer and Shallit~\citep{Shallit2021} considered it more natural to study string attractors of prefixes of infinite sequences instead of particular classes of finite factors of those sequences. This was then formalized in terms of {\em string attractor profile function} in~\citep{Restivo2022}.
The classical notions in combinatorics on words that somehow measure repetitiveness are the factor complexity and the recurrence function.
The string attractor profile function builds a bridge between them, as described in~\citep{Restivo2022}. Its behaviour was studied for linearly recurrent sequences and for automatic sequences in~\citep{Shallit2021}. Ibidem, the authors determined the values of string attractor profile function for the period-doubling sequence, the Thue-Morse sequence, the Tribonacci sequence, and the powers of two sequence (see also~\citep{Kociumaka2021}).
The string attractor profile function of standard Sturmian sequences was determined and its properties for fixed points of morphisms were studied in~\citep{Restivo2022}.

In this paper we determine string attractors of all factors of episturmian sequences. In a preliminary paper~\citep{Mantaci2019}, the authors announced the form of string attractor of circularly balanced epistandard words and promised to detail the proof in the full version (using palindromic closures). However to our knowledge, the full version does not contain such results. Our construction of attractors of factors of episturmian sequences is very simple, it is also based on palindromic closures. The size of the obtained string attractors of factors of episturmian sequences is the smallest possible, it equals the number of distinct letters contained in the factor. It provides string attractors for all factors of Sturmian sequences unlike the construction in~\citep{Mantaci2021}, which works only for finite standard Sturmian words. Our attractors are different because they do not contain consecutive positions in general. It is a straightforward consequence of our result that the string attractor profile function is eventually constant for any episturmian sequence.

\section{Preliminaries}
\label{Section_Preliminaries}
An \textit{alphabet} $\A$ is a finite set of symbols called \textit{letters}.
A \textit{word} over $\A$ of \textit{length} $n$ is a string $u = u_0 u_1 \cdots u_{n-1}$, where $u_i \in \A$ for all $i \in \{0,1, \ldots, n-1\}$. We let $|u|$ denote the length of $u$ and $\overline{u}=u_{n-1}\cdots u_1 u_0$. If $u=\overline{u}$, then $u$ is called a \textit{palindrome}.
The set of all finite words over $\A$ together with the operation of concatenation forms a monoid, denoted $\A^*$.
Its neutral element is the \textit{empty word} $\varepsilon$ and we write $\A^+ = \A^* \setminus \{\varepsilon\}$.
If $u = xyz$ for some $x,y,z \in \A^*$, then $x$ is a \textit{prefix} of $u$, $z$ is a \textit{suffix} of $u$ and $y$ is a \textit{factor} of $u$.

A \textit{sequence} over $\A$ is an infinite string $\uu = u_0 u_1 u_2 \cdots$, where $u_i \in \A$ for all $i \in \N$. 
% The notation $\A^{\mathbb{N}}$ stands for the set of all sequences over $\A$. 
We always denote sequences by bold letters. 

A sequence $\uu$ is \textit{eventually periodic} if $\uu = vwww \cdots = v(w)^\omega$ for some $v \in \A^*$ and $w \in \A^+$. If $\uu$ is not eventually periodic, then it is \textit{aperiodic}.
A \textit{factor} of $\uu = u_0 u_1 u_2 \cdots$ is a word $y$ such that $y = u_i u_{i+1} u_{i+2} \cdots u_{j-1}$ for some $i, j \in \N$, $i \leq j$. If $i=j$, then $y=\varepsilon$.
In the context of string attractors, the set $\{i, i+1,\dots, j-1\}$ is called an \textit{occurrence} of the factor $y$ in $\uu$. (Usually, only the number $i$ is called an occurrence of $y$ in $\uu$.)
If $i=0$, the factor $y$ is a \textit{prefix} of $\uu$.
If each factor of $\uu$ occurs at least twice in $\uu$, the sequence $\uu$ is \textit{recurrent}.
% If each factor of $\uu$ occurs moreover with bounded distances, the sequence $\uu$ is \textit{uniformly recurrent}. 

The \textit{language} $\mathcal{L}(\uu)$ of a sequence $\uu$ is the set of all its factors.
${\mathcal L}(\uu)$ is \textit{closed under reversal} if for each factor $w$ of $\uu$, the language ${\mathcal L}(\uu)$ contains also $\overline{w}$. 
% We also define $\mathcal{L}(\uu)^+ = \mathcal{L}(\uu) \setminus \{ \varepsilon \}$.
A factor $w$ of $\uu$ is \textit{left special} if $aw, bw$ are in $\mathcal{L}(\uu)$ for at least two distinct letters $a,b \in \A$. The \textit{factor complexity} of a sequence $\uu$ is the mapping $\mathcal{C}_\uu: \N \to \N$ defined by
$\mathcal{C}_\uu(n) = \# \{w \in \mathcal{L}(\uu) : |w| =  n \}$.
The factor complexity of an aperiodic sequence $\uu$ satisfies $\mathcal{C}_{\uu}(n) \ge n+1$ for all $n \in \N$. The aperiodic sequences with the lowest possible factor complexity $\mathcal{C}_{\uu}(n) = n+1$ are called \textit{Sturmian sequences}. Clearly, all Sturmian sequences are defined over a binary alphabet, e.g., $\{ {\tt 0,1} \}$.

Kempa and Prezza~\citep{KempaPrezza2018} introduced the notion of string attractor.  
A {\em string attractor} (or {\em attractor} for short) of a word $w=w_0w_1\cdots w_{n-1}$, where $w_i\in \mathcal A$, is a set $\Gamma \subset \{0,1,\dots, n-1\}$ such that every factor of $w$ has an occurrence containing at least one element of $\Gamma$. For instance, $\Gamma=\{1,3\}$ is an attractor of $w=\tt 0\underline{1}0\underline{0}10$ (it corresponds to the underlined positions). It is the smallest possible attractor since each attractor necessarily contains occurrences of all distinct letters of the factor.

\section{Palindromic closures and episturmian sequences}
\label{sec:sturmian}
\begin{definition}
Let $w$ be a word and $a$ a letter, then $(wa)^{(+)}$ is the shortest palindrome having $wa$ as prefix. 
\end{definition}
It follows immediately from the definition that $(wa)^{(+)}=wa\overline{w}$ if $w$ does not contain $a$. Otherwise, $(wa)^{(+)}=vp\overline{v}$, where $w=vp$ and $p$ is the longest palindromic suffix of $w$ preceded by $a$.
\begin{example}
Let $w=\tt 000$, then $(w{\tt 0})^{(+)}=w\tt 0=\tt 0000$ and $(w{\tt 1})^{(+)}=w{\tt 1}w=\tt 0001000$.
For $v=\tt 01101$, we have $(v{\tt 0})^{(+)}=\tt 011010110$ and $(v{\tt 1})^{(+)}=\tt 0110110$.
\end{example}
\begin{definition}\label{def:palclosure}
Let $\Delta=\delta_0\delta_1\delta_2\cdots$ with $\delta_i \in \mathcal A$ and define $w_0=\varepsilon$ and
$w_{n+1}=(w_{n}\delta_{n})^{(+)}$ for all $n\in \mathbb N$. Then we denote $\uu(\Delta)=\lim_{n\to \infty} w_n$, i.e., $\uu(\Delta)$ is a unique sequence having $w_n$ as prefix for each $n \in \mathbb N$, and we call $\Delta$ the \textit{directive sequence} of $\uu(\Delta)$.
\end{definition}
 
\begin{definition}
Let $\uu$ be a sequence whose language is closed under reversal and such that for each length $n$ it contains at most one left special factor. Then $\uu$ is called an \textit{episturmian sequence}. 
An episturmian sequence is \textit{standard} if all left special factors are prefixes. 
\end{definition}
Sturmian sequences correspond to aperiodic binary episturmian sequences.
It is well-known that for each episturmian sequence there exists a unique standard episturmian sequence with the same language. Since we are interested in the language of episturmian sequences, it suffices to consider only standard episturmian sequences~\citep{DrJuPi2001}. 
For the study of attractors, the construction of episturmian sequences by palindromic closures seems to be handy. It was introduced by Droubay, Justin and Pirillo~\citep{DrJuPi2001}.

\begin{theorem} 
Let $\uu$ be a standard episturmian sequence over $\mathcal A$. Then $\uu=\uu(\Delta)$ for a unique sequence $\Delta=\delta_0\delta_1\delta_2\cdots$ with $\delta_i \in \mathcal A$.
\end{theorem}
    In the sequel we will always assume without loss of generality that $\Delta$ is defined over ${\mathcal A}=\{\tt 0,1,\dots, d-1\}$ and the first letter of $\Delta$ is $\tt 0$, the second distinct letter in $\Delta$ is $\tt 1$, etc.

\begin{example}
\label{ex:FiboDef}
The most famous standard Sturmian sequence is the \textit{Fibonacci sequence}
$$
\mathbf f = \mathbf f(\Delta)\,,
$$
where $\Delta=(\tt 01)^{\omega}$.
The first six prefixes of $\mathbf f$ read:
$$\begin{array}{rcll}
  w_0&=&\varepsilon \\
  w_1&=&\tt 0\\
  w_2&=&\tt 010 \\
  w_3&=&\tt 010010\\
  w_4&=&\tt 01001010010\\
  w_5&=&\tt 0100101001001010010\,.
\end{array}$$
\end{example}

\section{Attractors of episturmian sequences}
In this section we determine attractors of all factors of episturmian sequences. We start with attractors of palindromic prefixes of standard episturmian sequences.
% \begin{theorem}\label{thm:standardAR}
% Let $\uu$ be a standard Arnoux-Rauzy sequence and let $\Delta$ be its directive sequence. 
% Let $w_n$ contain the letters $\tt 0,1,\dots, d-1$ for $d\in \mathbb N, d\geq 1$. Then $w_n$ has an attractor of size $d$.
% Namely, the set $\{m_0, m_1, \dots, m_{d-1}\}$ is an attractor of $w_n$, where $m_j$ is the length of the longest palindromic prefix of $w_n$ followed by $\tt j$, i.e., $m_j=\max\{|w_i| \ : \ 0 \leq i <n \ \text{and} \ w_i{\tt j} \ \text{is a prefix of} \ w_n\}.$
% \end{theorem}

\begin{theorem}\label{thm:standardAR}
Let $v$ be a non-empty palindromic prefix of a standard episturmian sequence. 
For every letter $a$ occurring in $v$,  denote 
$$m_a=\max\{|p| \, : \, p \text{\  is a palindrome  and }  pa \text{ is a prefix of} \ v\}.$$
Then $\Gamma = \{m_a: a \text{ occurs in  } v\} $ is an attractor of $v$ and its size is minimal.
\end{theorem}

\begin{proof}
By the definition of palindromic closure, each palindromic prefix $v$ is equal to $w_n$ for some $n \in \mathbb N$.
We will prove the statement by mathematical induction on $n$. Let us recall that we index positions from $0$, i.e., $v=v_0 v_1\dots v_{|v|-1}$.
\begin{itemize}
    \item For $n=1$ we have $w_1=\tt 0$ and its attractor equals $\{0\}$. The longest palindromic prefix of $w_1$ followed by $\tt 0$ is equal to $w_0=\varepsilon$ and its length satisfies $|w_0|=0$.
    \item For $n\geq 2$ we assume that $w_{n-1}$ has an attractor of the form from the statement. We have $w_n=(w_{n-1}{\tt j})^{(+)}$ for some $\tt j \in \{\tt 0,1,\dots, d-1\}$. The following three situations may occur:
    \begin{enumerate}
        \item $w_n=w_{n-1}\tt j$: According to the definition of palindromic closure, this happens only for $\tt j=\tt 0$ and $w_{n-1}={\tt 0}^{\ell}$ for some $\ell \in \mathbb N, \ell \geq 1$. The longest palindromic prefix of $w_n={\tt 0}^{\ell+1}$ followed by $\tt 0$ is $w_{n-1}={\tt 0}^{\ell}$. We have thus $|w_{n-1}|=\ell$ and indeed $\{\ell\}$ is an attractor of $w_n$.
        \item $w_n=w_{n-1}{\tt j}w_{n-1}$: By the definition of palindromic closure, this happens only in case when $w_{n-1}$ contains only letters $\tt 0,\dots, j-1$. It follows from the form of $w_n$ that the longest palindromic prefix of $w_n$ followed by $\tt i$, where ${\tt i} \in \{\tt 0,\dots, j-1\}$, is the same as in $w_{n-1}$. Moreover, the longest palindromic prefix of $w_n$ followed by $\tt j$ is $w_{n-1}$. Let us explain that $\{m_{\tt 0}, \dots, m_{\tt j-1}, m_{\tt j}\}$ as defined in the statement is an attractor of $w_n$: each factor of $w_n$ either contains the letter $\tt j$, i.e., it has an occurrence containing the position $m_{\tt j}=|w_{n-1}|$, or it is contained in $w_{n-1}$, which is a prefix of $w_n$, and by induction assumption it has an occurrence containing $m_{\tt i}$ for some $\tt i \in \{\tt 0,\dots, j-1\}$.
        \item $w_n=w_{n-1}{\tt j}u$ for some $u \not =\varepsilon$ and $u \not =w_{n-1}$: Assume $w_n$ contains $k \leq d$ letters. We want to prove that $\{m_{\tt 0}, \dots, m_{\tt j}, \dots, m_{\tt k-1}\}$, as defined in the statement, is an attractor of $w_n$. Since the longest palindromic prefix of $w_n$ followed by $\tt i$, where $\tt i \in \{\tt 0,\dots, k-1\}$, $\tt i\not =\tt j$, is the same as in $w_{n-1}$, we know by induction assumption that $\{m_{\tt 0}, \dots, m'_{\tt j},\dots, m_{\tt k-1}\}$ is an attractor of $w_{n-1}$, where we replaced $m_{\tt j}$ by $m'_{\tt j}=|w_\ell|$, where $w_\ell$ is the longest palindromic prefix of $w_{n-1}$ followed by $\tt j$.
        By the definition of palindromic closure we have 
        
        \begin{equation}\label{eq:w_n}
       w_n=\underbrace{\overline{u}{\tt j}w_\ell}_{w_{n-1}}{\tt j}u=\overline{u}{\tt j}\underbrace{w_\ell{\tt j}u}_{w_{n-1}}\,.
        \end{equation}
        
       Then each factor of $w_n$ either has an occurrence containing the position $|w_{n-1}|$, i.e., crossing the second $\tt j$ in the expression~\eqref{eq:w_n}, or is entirely contained in $w_{n-1}$. In the latter case, it has an occurrence crossing the attractor of $w_{n-1}$. However, it has no occurrence containing the position $m'_{\tt j}=|w_\ell|$ because, according to~\eqref{eq:w_n}, it would then also have an occurrence containing the position $m_{\tt j}=|w_{n-1}|$. To sum up, we have proved that each factor of $w_n$ has an occurrence crossing $\{m_{\tt 0}, \dots, m_{\tt j}, \dots, m_{\tt k-1}\}$.
    \end{enumerate}
    \end{itemize}
\end{proof}
\begin{example}
Let $\Delta=({\tt 012})^{\omega}$, i.e., $\uu=\uu(\Delta)$ is the Tribonacci sequence. 
We underline the positions of the attractor from Theorem~\ref{thm:standardAR} in the first five non-empty palindromic prefixes: 
$$\begin{array}{rcll}
  w_0&=&\varepsilon \\
  w_1&=&\underline{\tt 0}\\
  w_2&=&\underline{\tt 0}\underline{\tt 1}\tt 0 \\
  w_3&=&\underline{\tt 0}\underline{\tt 1}\tt 0\underline{\tt 2}\tt 010 \\
  w_4&=&\tt 0\underline{\tt 1}\tt 0\underline{\tt 2}\tt 010\underline{\tt 0}\tt 102010\\
  w_5&=&\tt 010\underline{\tt 2}\tt 010\underline{\tt 0}\tt 102010\underline{\tt 1}\tt 020100102010\,.
\end{array}$$
Indeed,
\begin{itemize}
    \item $w_1$ contains only $\tt 0$ and its longest palindromic prefix followed by $\tt 0$ is equal to $w_0=\varepsilon$, therefore its attractor equals $\{0\}$;
    \item $w_2$ contains $\tt 0,1$ and its longest palindromic prefix followed by $\tt 0$, resp., $\tt 1$, is equal to $w_0$, resp., $w_1=\tt 0$, i.e., $\{0,1\}$ is an attractor of $w_2$;
    \item $w_3$ contains $\tt 0,1,2$ and its longest palindromic prefix followed by $\tt 0$, resp., $\tt 1$, resp., $\tt 2$, is equal to $w_0$, resp., $w_1=\tt 0$, resp., $w_2=\tt 010$, i.e., $\{0,1,3\}$ is an attractor of $w_3$;
    \item $w_4$ contains $\tt 0,1,2$ and its longest palindromic prefix followed by $\tt 0$, resp., $\tt 1$, resp., $\tt 2$, is equal to $w_3=\tt 0102010$, resp., $w_1=\tt 0$, resp., $w_2=\tt 010$, i.e., $\{7,1,3\}$ is an attractor of $w_4$;
    \item $w_5$ contains $\tt 0,1,2$ and its longest palindromic prefix followed by $\tt 0$, resp., $\tt 1$, resp., $\tt 2$, is equal to $w_3$, resp., $w_4=\tt 01020100102010$, resp., $w_2$, i.e., $\{7,14,3\}$ is an attractor of $w_5$.
\end{itemize}
In simple terms, the attractor of $w_n$ is equal to $\{|w_{n-3}|, |w_{n-2}|, |w_{n-1}|\}$ for $n \geq 3$. Moreover, it is not difficult to determine that $|w_n|=\frac{T_{n+3}+T_{n+1}-3}{2}$, where $T_{n}=T_{n-1}+T_{n-2}+T_{n-3}$ for all $n \in \mathbb N, n \geq 3$, and $T_0=0, T_1=1, T_2=1$. 
This attractor of $w_n$ is different from the one proposed in~\citep{Shallit2021}.

We can easily observe that when increasing $n$ by one, either one new element appears in the attractor or one element of the attractor changes. This holds in general for palindromic prefixes of standard episturmian sequences. 
\end{example}

\begin{remark}\label{rem:mirror}
Since $w_n$ is a palindrome, the mirror image $\overline{\Gamma}=\{|w_n|-1-m_a \ : \ a \ \text{occurs in}\  w_n\}$ of the attractor $\Gamma=\{m_a \ : \ a\ \text{occurs in}\ w_n\}$ of $w_n$ is an attractor of $w_n$, too. 
\end{remark}

\begin{theorem}\label{thm:attractorAR}
Let $\uu$ be an episturmian sequence. Each factor of $\uu$ containing $d$ distinct letters has an attractor of size $d$.
\end{theorem}
\begin{proof}
We will use the fact that  
there exists a unique directive sequence $\Delta=\delta_0\delta_1\delta_2\cdots$ such that ${\mathcal L}(\uu)={\mathcal L}(\uu(\Delta))$. 
If $d=1$, then the statement evidently holds. Consider a factor $w$ of $\uu$ containing $d \geq 2$ distinct letters. 
Let $n \in \mathbb N$ be minimal such that $w_n$ contains $w$. Either $w=w_n$, then $w$ has an attractor of size $d$ by Theorem~\ref{thm:standardAR}. Or $w \not =w_n$.
Then there are two cases to be treated:
\begin{enumerate}
\item If $w_{n-1}$ does not contain $\delta_n$, then
$$w_n=(w_{n-1}\delta_n)^{(+)}=w_{n-1}\delta_nw_{n-1}.$$ 
Then each occurrence of $w$ contains the occurrence of $\delta_n$: more precisely $|w_{n-1}|$ (otherwise $w$ would be contained in $w_{n-1}$). Consider an arbitrary occurrence of $w$ in $w_n$.
Find minimal indices $i$ and $j$, where $ 0\leq  i,j \leq n-1$, such that $w$ is contained in 
$w_i\delta_n w_j$. 
By minimality of $i$ and $j$, it is clear that $w$ contains all proper palindromic prefixes of $w_j$ together with the following letter and all proper palindromic suffixes of $w_i$ together with the preceding letter. 
Therefore on one hand, the considered occurrence of $w$ in $w_n$ contains the attractor of $w_j$ given in Theorem~\ref{thm:standardAR} shifted by $|w_{n-1}|+1$ and the mirror image of the attractor of $w_i$ given in Theorem~\ref{thm:standardAR} shifted by $|w_{n-1}|-|w_i|$ (see Remark~\ref{rem:mirror}). 
Assume WLOG $j \geq i$. On the other hand, each factor of $w$ either has an occurrence crossing $\delta_n$ in $w_n$, i.e., containing the position $|w_{n-1}|$, or it is contained in $w_j$ ($w_i$ is a prefix of $w_j$), in which case it has an occurrence crossing the attractor of $w_j$ shifted by $|w_{n-1}|+1$. Since $w$ has $d$ distinct letters, $w_j$ has $d-1$ distinct letters and the attractor of $w_j$ is of size $d-1$ by Theorem~\ref{thm:standardAR}. Altogether, it implies existence of an attractor of size $d$ for $w$.

\item If $w_{n-1}$ contains $\delta_n$, then by the definition of palindromic closure we have
\begin{equation}\label{eq:closure}
w_n=(w_{n-1}\delta_n)^{(+)}=w_{n-1}\delta_n \overline{u}=u\delta_n w_{n-1}=u\underline{\delta_n} w_k \underline{\delta_n} \overline{u},
\end{equation}
where $u$ is a non-empty word and $w_k$ is the longest palindromic prefix of $w_{n-1}$ followed by $\delta_n$ ($w_k$ may be empty). Then each occurrence of $w$ contains both underlined occurrences of $\delta_n$: more precisely $|u|$ and $|w_{n-1}|$ (otherwise $w$ would be contained in $w_{n-1}$). Consider an arbitrary occurrence of $w$ in $w_n$.
Find minimal indices $i$ and $j$, where $k < i,j \leq n-1$, such that $w$ is contained in 
$$v=\underbrace{y\underline{\delta_n} w_k}_{w_i} \underline{\delta_n} z=y\underline{\delta_n} \underbrace{w_k \underline{\delta_n} z}_{w_j}\,,$$ 
where the underlined positions correspond to the ones underlined in~\eqref{eq:closure}. 
By minimality of $i$ and $j$, it is clear that $w$ contains all proper palindromic prefixes of $w_j$ together with the following letter and all proper palindromic suffixes of $w_i$ together with the preceding letter. 
Therefore on one hand, the considered occurrence of $w$ in $w_n$ contains the attractor of $w_j$ given in Theorem~\ref{thm:standardAR} shifted by $|u|+1$ and the mirror image of the attractor of $w_i$ given in Theorem~\ref{thm:standardAR} shifted by $|w_{n-1}|-|w_i|$ (see Remark~\ref{rem:mirror}). 
Assume WLOG $j \geq i$. On the other hand, each factor of $w$ either has an occurrence crossing the second underlined $\delta_n$ in $w_n$, i.e., containing the position $|w_k|+|u|+1$, or it is contained in $w_j$ ($w_i$ is a prefix of $w_j$). Hence it has an occurrence crossing the attractor of $w_j$ shifted by $|u|+1$, which is of size $d$ by Theorem~\ref{thm:standardAR}.
\end{enumerate}
\end{proof}
Restivo et al.~\citep{Restivo2022} defined the \textit{string attractor profile function} of a sequence $\uu$
as a map $s_\uu: \mathbb N \to \mathbb N$ satisfying $s_\uu(n)=$ the size of a smallest string attractor of the prefix of length $n$ of $\uu$. By Theorem~\ref{thm:attractorAR} episturmian sequences have an eventually constant string attractor profile function.

\begin{corollary}
Let $\uu = u_0u_1u_2\cdots $ be an episturmian sequence over $\mathcal{A}$. Then $$s_\uu(n)=\#\{a \in \mathcal{A}: a \text{ occurs in } u_0u_1\cdots u_{n-1}  \}.$$ 
In particular, the  string attractor profile function is eventually constant. 
\end{corollary}

\section{Acknowledgements}
I would like to thank to Edita Pelantová and Francesco Dolce for their careful reading and several improvement proposals.
This work was supported by the Ministry of Education, Youth and Sports of the Czech Republic through the project \\ CZ.02.1.01/0.0/0.0/16\_019/0000778.

\end{document}